\def\myMarginsize{1in}
\def\myStretch{1.1}

\documentclass[10pt]{article}

\usepackage{bbm}
\usepackage{caption}
\usepackage{rotating}
\usepackage[maxfloats=100]{morefloats}
\usepackage{listings}
\usepackage{booktabs}
\usepackage{xcolor}
\usepackage[title]{appendix}
\usepackage{makecell}
\usepackage{multirow}
\usepackage[symbol]{footmisc}
\usepackage{longtable}
\usepackage{subcaption}
\usepackage[numbers]{natbib}
\usepackage{geometry}
\usepackage{verbatim}
\usepackage{graphicx}
\usepackage{enumitem}
\usepackage{comment}
\usepackage{subfiles}
\usepackage{todonotes}
\usepackage{tikz-qtree,tikz-qtree-compat}
\usepackage{scrextend}
\usepackage{framed}
\usepackage{placeins}
\usepackage{setspace}
\usepackage[ruled,vlined]{algorithm2e}
\usepackage{mathtools,collcell,eqparbox}
\usepackage{multicol}
\usepackage{pgfplots}
\usepackage{tikz}
\usepackage{courier}
\usepackage{url}
\usepackage{color}
\usepackage{authblk}
\usepackage{amssymb}
\usepackage{amsthm}
\usepackage{dsfont}
\usepackage{amsmath}
\usepackage[pagebackref = true]{hyperref}

\hypersetup{
  colorlinks = true,
  linkcolor = teal,
  anchorcolor = teal,
  citecolor = teal,
  filecolor = teal,
  urlcolor = teal
}

\pgfplotsset{compat=1.15}
\numberwithin{equation}{section}

\makeatletter
\newcounter{BMatrix}

\newcommand{\setmaxwd}[1]{%
  \eqmakebox[BM-\theBMatrix][\BMalign]{$#1$}%
}
\MHInternalSyntaxOn

\MHInternalSyntaxOff
\makeatother

\newtheorem{theorem}{Theorem}
\newtheorem{assumption}{Assumption}

\newtheorem{proposition}{Proposition}[section]

\newtheorem{lemma}{Lemma}
\theoremstyle{definition}

\newtheorem{example}{Example}
\theoremstyle{remark}
\newtheorem{remark}{Remark}
\newtheorem*{remark*}{Remark}

\newcommand{\A}{\mathbb{A}}
\newcommand{\W}{\mathbb{W}}
\newcommand{\R}{\mathbb{R}}

\newcommand{\X}{\mathbb{X}}

\newcommand{\ra}{\rightarrow}

\newcommand{\ua}{\uparrow}

\newcommand{\cd}{\cdot}

\newcommand{\law}{\mathrm{Law}}
\newcommand{\mrm}[1]{\mathrm{#1}}

\newcommand{\cA}{\mathcal{A}}
\newcommand{\cB}{\mathcal{B}}

\newcommand{\cE}{\mathcal{E}}

\newcommand{\cP}{\mathcal{P}}

\newcommand{\cT}{\mathcal{T}}
\newcommand{\cU}{\mathcal{U}}

\newcommand{\cW}{\mathcal{W}}
\newcommand{\cX}{\mathcal{X}}

\newcommand{\set}[1]{\left\{{#1}\right\}}
\newcommand{\setcond}[2]{\left\{ #1 : #2 \right\}}

\newcommand{\norm}[1]{\left\|#1\right\|}

\newcommand{\abs}[1]{\left|#1\right|}
\newcommand{\sqbk}[1]{\left[ #1 \right]}

\newcommand{\crbk}[1]{\left( #1 \right)}

\newcommand{\Wass}{\mathsf{W}}

\newcommand{\Bkap}{\mathcal{B}_{\kappa}}

\title{Lipschitz Regularity in \\ Wasserstein Robust Stochastic Optimal Control}
\author[1]{Shengbo Wang}
\author[2]{Jose Blanchet}
\affil[1]{Daniel J. Epstein Department of Industrial and Systems Engineering\\ University of Southern California}
\affil[2]{Management Science and Engineering\\ Stanford University}
\date{June 2026}
\begin{document}

\maketitle

\begin{abstract}
Robust Markov decision processes provide a principled framework for protecting sequential decision-making against transition-law misspecification and have attracted substantial recent research interest. As in non-robust stochastic optimal control, an important question is whether the robust value function is sufficiently regular for approximation and learning. This paper studies Lipschitz regularity of optimal value functions for Wasserstein robust stochastic optimal control on possibly unbounded Polish state spaces under an infinite-horizon discounted reward criterion. We consider two robustness formulations: a kernel-robust model, in which the adversary perturbs the next-state distribution within a Wasserstein ball around a nominal transition kernel, and a noise-robust model, in which the adversary perturbs the driving noise of the state transition dynamics. In the non-robust setting, Lipschitz rewards and Lipschitz transition dynamics do not, in general, imply a Lipschitz value function. In contrast, we show that in these Wasserstein robust formulations, Lipschitz assumptions on the model primitives yield Lipschitz robust value functions. Thus, Wasserstein robustness not only protects against misspecification but also regularizes the Bellman fixed point, providing stability relevant to discretization, value-function approximation, estimation, and learning.
\end{abstract}

\section{Introduction}
Markov decision processes (MDPs) and stochastic control provide a canonical framework for sequential decision-making under uncertainty. In the classical formulation, the transition law is treated as part of the model specification. Under standard assumptions, dynamic programming reduces the control problem to solving a Bellman equation that characterizes the optimal values and policies \citep{bertsekas1978stochastic}. In data-driven or simulation-based settings, however, the transition law is rarely known exactly. It may be estimated from limited data, approximated by a simulator, or used as a simplified representation of the true dynamics. Consequently, potential model misspecification can harm the performance of the resulting policy. A risk-averse response to such misspecification is to evaluate policies against a family of plausible models rather than a single nominal model, leading to robust MDP formulations \citep{iyengar2005robust,wiesemann2013robust,wang2025foundation}.

This paper studies an important regularity question for general state space Wasserstein robust stochastic optimal control. In applications of stochastic optimal control, the optimal value function is of central interest beyond certifying Markov optimal policies. It drives many core aspects of control and learning algorithm design: it is the object approximated in discretization, fitted value iteration, model-based reinforcement learning, and policy improvement. The regularity of the value function is therefore operationally important, because it usually controls how state approximation, statistical error, and model error propagate through dynamic programming \citep{hinderer2005lipschitz,dufour2012approximation,asadi2018lipschitz,wang2026fast_convergenece}.  A natural first guess is that Lipschitz rewards and Lipschitz dynamics should imply a Lipschitz value function.  The following simple example shows that this implication is false.

\begin{example}\label{exmpl:holder}
Let the state space $\X=[0,1]$ and the action space $\A=\{a_0\}$. Fix $\alpha\in(0,1)$ and set $\gamma=2^{-\alpha}$. Consider the system dynamics induced by the transition function $f(x)=\min\{2x,1\}$; i.e. the MDP has transition kernel $ P(\cdot| x,a_0)=\delta_{f(x)}$. Also, we use the reward
\[
r(x):=x^\alpha-\gamma f(x)^\alpha
    =
    \begin{cases}
        0, & 0\le x\le \frac12,\\
        x^\alpha-2^{-\alpha}, & \frac12\le x\le 1.
    \end{cases}
\]
Then $r\ge 0$ is Lipschitz on $[0,1]$, and $f$ is $2$-Lipschitz. 

Since there is only one action, the classical dynamic programming theory \citep{bertsekas1978stochastic} implies that the optimal value function $V^*$ is the unique bounded solution of $V^*(x)=r(x)+\gamma V^*(f(x))$.
By construction, $x^\alpha
    =
    r(x)+\gamma f(x)^\alpha,$
so $V^*(x)=x^\alpha .$
Therefore, $V^*$ is $\alpha$-H\"older but not Lipschitz. 
\end{example}

The example isolates a structural obstruction. The loss of Lipschitz regularity is not caused by stochasticity, multiple actions, or non-compactness. Rather, it arises from the Bellman recursion itself: the value function is composed with the dynamics, and expansion in the dynamics (cf.\ the $2$-Lipschitz part of $f$ in Example~\ref{exmpl:holder}) can destroy Lipschitz continuity of the fixed point.

We show that Wasserstein distributional robustness changes this conclusion. We study infinite-horizon control on possibly unbounded Polish state spaces, with a discounted objective and non-negative rewards. For such problems, the optimal robust value functions are typically characterized as minimal non-negative solutions of robust Bellman equations \citep{iyengar2005robust,wang2025foundation}. We analyze two adversarial formulations.
\begin{itemize}[leftmargin = *]
    \item In the kernel-robust model, the adversary directly perturbs the next-state distribution. Given a nominal transition kernel $P_0$, the ambiguity set is $\cP_\delta(x,a):=\set{\nu:\Wass_p(\nu,P_0(\cdot\mid x,a))\leq\delta}.$ Under SA-rectangularity, the robust optimal value function is the fixed point of the Bellman operator in \eqref{eqn:kernel_operator}; this is the robust dynamic programming principle of \citet{iyengar2005robust} for Wasserstein ambiguity sets.
\item In the noise-robust model, we consider controlled stochastic dynamics of the form $X_{t+1}=f(X_t,A_t,W_{t+1})$. Under the nominal model, the driving noise sequence $\set{W_t:t\geq 1}$ is i.i.d.\ with common law $\psi_0$. The adversary perturbs the marginal law of the driving noise within the Wasserstein ball $\cU_{\delta}:=\set{\psi:\Wass_p(\psi,\psi_0)\leq\delta}.$ Depending on whether the adversary observes the realized current action of the controller, this gives rise to the current-action-aware (CAA) and current-action-unaware (CAU) formulations in \citet{wang2025statistical}. The corresponding Bellman operators are given in \eqref{eqn:caa_operator} and \eqref{eqn:cau_operator}.
    
\end{itemize}

The main contribution is to show that, for every positive Wasserstein radius, the Wasserstein robust optimal value functions under these formulations are locally Lipschitz under natural Lipschitz and drift conditions, even though the corresponding non-robust value function (i.e., when $\delta = 0$) may fail to be Lipschitz. This result gives a complementary interpretation of Wasserstein robust control. While Wasserstein ambiguity is typically introduced to protect against model misspecification risk \citep{blanchet2019modelrisk,esfahani2018data,wang2025statistical}, in the dynamic setting, it also restores the Lipschitz regularity of the Bellman fixed point. This additional regularization effect is consequential for learning and approximation, as it provides additional stability and smoothness to control discretization, estimation, and value-function approximation errors.

We remark that this regularization effect is not a generic consequence of robustness, but specific to Wasserstein ambiguity whose dual representation involves a spatial inf-convolution; see Proposition~\ref{prop:wasserstein_lower_duality}. In particular, for KL-divergence-based ambiguity set $\{\nu:D_{\mrm{KL}}(\nu\|P_0(\cdot| x,a))\le\delta\}$, deterministic dynamics leave Example~\ref{exmpl:holder} unchanged: if $P_0(\cdot| x,a)=\delta_{f(x)}$, the adversary must choose $\nu=\delta_{f(x)}$. Thus, the robust Bellman operator coincides with the non-robust one in that example, and the value function remains non-Lipschitz.

\subsection{Literature Review}

\paragraph{Robust and Wasserstein robust Markov decision processes.}
Robust MDPs replace a nominal transition model by an uncertainty set and use dynamic programming to compute worst-case values.  \citet{iyengar2005robust} and \citet{nilim2005robust} established robust dynamic programming under SA-rectangular transition uncertainty, \citet{wiesemann2013robust} developed tractable uncertainty sets with probabilistic guarantees.  More closely related to our formulation, \citet{yang2017convex} studied Wasserstein distributionally robust MDPs through convex reformulations, \citet{yang2021wasserstein} and \citet{wang2025statistical} developed data-driven Wasserstein robust stochastic control, and \citet{grandclement2021first} proposed first-order methods for Wasserstein robust MDPs. 

\paragraph{Wasserstein robustness and sensitivity.}
Wasserstein ambiguity uses optimal-transport geometry to compare probability laws \citep{villani2009optimal}.  In static distributionally robust optimization, \citet{esfahani2018data}, \citet{gao2023distributionally}, and \citet{blanchet2019modelrisk} develop duality, tractability, and performance guarantees for Wasserstein ambiguity sets.  Sensitivity of Wasserstein robust objectives is studied by \citet{bartl2021sensitivity}.  Our analysis uses related duality and sensitivity ideas.

\paragraph{Regularity and approximation of value functions.}
Regularity and stability of value functions are standard inquiries in dynamic programming.  \citet{muller1997value} and \citet{zhou2026robustnessmodelapproximationmodel} study the dependence of value functions on transition probabilities and rewards, \citet{hinderer2005lipschitz} gives conditions for Lipschitz continuity in MDPs, and \citet{dufour2012approximation} studies approximation for general state spaces.  In learning-based contexts, smoothness of the optimal value function has been used to control value-estimation error and policy regret in model-based reinforcement learning and data-driven control \citep{asadi2018lipschitz,wang2026fast_convergenece,wang2026QML}, leading to efficient algorithms and learning guarantees.

\section{Wasserstein Robust Control on Unbounded Spaces}
\label{section:setup}

We consider Polish state, action, and noise spaces $\X,\ \A,\ \W$ with metrics $d_\X,d_\A,d_\W$ and Borel $\sigma$-algebras $\cX,\ \cA,\ \cW$, respectively. In addition, we assume throughout that $(\X,d_\X)$ and $(\W,d_\W)$ are proper Polish spaces, so that closed bounded subsets of $\X$ and $\W$ are compact.  This ensures the compactness of relevant Wasserstein balls in the topology of weak convergence. We let $m\set{\cU\ra\cB(\R)}$ denote measurable functions from $(\mathbb{U},\cU)$ to $(\R,\cB(\R))$.

We fix a reference state $x_0\in\X$, a reference noise point $w_0\in\W$, and power $p\in[1,\infty)$ for Wasserstein distances throughout the paper. For $\mu\in\cP(\cX)$, write $\mu\in\cP_p(\cX)$ if
$\int_\X d_\X(x,x_0)^p\mu(dx)<\infty.$ The class $\cP_p(\cW)$ is defined analogously. Then, for $\mu,\nu\in\cP_p(\cX)$, define the Wasserstein-$p$ distance as
\[
    \Wass_p(\mu,\nu)
    :=
    \inf_{\pi\in\Pi(\mu,\nu)}
    \left(\int_{\X\times\X}d_\X(x,y)^p\pi(dx,dy)\right)^{1/p},
\]
and the same notation is used on $\W$ with $d_\W$ in place of $d_\X$.

We consider the $\gamma$-discounted setting with $\gamma\in(0,1)$ and non-negative reward functions $r$.  Hence, by a verification argument, the infinite-horizon robust value function is usually the minimal non-negative solution of the robust Bellman equation. In this paper, we consider two main types of Wasserstein distributionally robust stochastic optimal control formulations, leading to three versions of the robust Bellman equation. 

\paragraph{Kernel-robust model \citep{iyengar2005robust}.} We consider a Wasserstein ambiguity model directly imposed on the transition kernel. Specifically, the adversary chooses the transition distribution of the next state from the ambiguity set
\[
    \cP_\delta(x,a):=
    \setcond{\nu\in\cP_p(\cX)}{\Wass_p(\nu,P_0(\cd|x,a))\le\delta}.
\]
In this setting, we consider a reward function $r:\X\times\A\times\X\ra \R_+$. The relevant Bellman operator is
\begin{equation}
\label{eqn:kernel_operator}
\cT_\delta v(x)
    :=
    \sup_{a\in\A}
    \inf_{\nu\in\cP_\delta(x,a)}
    \int_\X \sqbk{r(x,a,y)+\gamma v(y)}\nu(dy).
\end{equation}
The non-robust Bellman operator is obtained by setting $\delta=0$.
\paragraph{Noise-robust model \citep{wang2025statistical}.} We also consider a state space dynamics formulation of a controlled Markov chain driven by the dynamics $f$ and an exogenous noise process $\set{W_t\in\W:t\ge 1}$. In this setting, we allow the reward to explicitly depend on the noise; i.e., $r:\X\times\A\times\W\ra \R_+$. The state transition is given by $X_{t+1}=f(X_t,A_t,W_{t+1}).$
Under the nominal model, the noise sequence is i.i.d. with common law $\psi_0\in\cP_p(\cW)$.  The adversary can perturb the marginal distributions of $\set{W_t:t\ge 1}$ within the ambiguity set
\[
    \cU_\delta:=\set{\psi\in\cP_p(\cW):\Wass_p(\psi,\psi_0)\le\delta}.
\]
Under this model, the adversary's observability of the current action will affect the dynamic programming equation and hence the stationary robust optimal policies.

\begin{itemize}[leftmargin = *]
    \item If the adversary can observe the realized current action, then the robust optimal value function is given by the fixed point of the current-action-aware (CAA) operator
\begin{equation}
\label{eqn:caa_operator}
    \cT_\delta^{\mrm{CAA}} v(x)
    :=
    \sup_{a\in\A}
    \inf_{\psi\in\cU_\delta}
    \int_\W \sqbk{r(x,a,w)+\gamma v(f(x,a,w))}\psi(dw).
\end{equation}

\item On the other hand, if the adversary cannot observe the realized current action, then we consider the current-action-unaware (CAU) operator
\begin{equation}
\label{eqn:cau_operator}
    \cT_\delta^{\mrm{CAU}} v(x)
    :=
    \sup_{\phi\in\cP(\cA)}
    \inf_{\psi\in\cU_\delta}
    \int_\W\int_\A
    \sqbk{r(x,a,w)+\gamma v(f(x,a,w))}
    \phi(da)\psi(dw).
\end{equation}
In this case, the optimal policy is usually not deterministic. 
\end{itemize}

Before we present our main results, we first establish some technical definitions and tools.

\subsection{Wasserstein Duality, Approximation, and Sensitivity}

Our main strategy for proving the existence and regularity of the minimal non-negative fixed points of \eqref{eqn:kernel_operator}, \eqref{eqn:caa_operator}, and \eqref{eqn:cau_operator} is based on a monotone Picard iteration argument, which approximates the solution from below. To control the regularity of the approximating iterates, we use the dual form of the Wasserstein robust functional. For these purposes, we need the following proposition. 

\begin{proposition}
\label{prop:wasserstein_lower_duality}
Let $(E,\cE)$ denote either $(\X,\cX)$ or $(\W,\cW)$.  Let $\mu\in\cP_p(\cE)$, $\delta>0$, and $h:E\to[0, +\infty]$ be lower semicontinuous with $\int hd\mu<\infty$.  Define $h^\lambda(z):=\inf_{y\in E}\set{h(y)+\lambda d(z,y)^p}$ for $\lambda\ge0.$
Then
\begin{equation}
\label{eqn:bm_metric_dual}
    \inf_{\nu\in\cP_p(E):\Wass_p(\nu,\mu)\le\delta}
    \int hd\nu
    =
    \sup_{\lambda\ge0}
    \left\{\int h^\lambda(z)\mu(dz)-\lambda\delta^p\right\}.
\end{equation}
Moreover, if $\set{h_n:n\ge 1}$ are non-negative lower semicontinuous with $h_n\ua h$ pointwise, then
\begin{equation}
\label{eqn:robust_monotone_continuity}
    \inf_{\nu:\Wass_p(\nu,\mu)\le\delta}\int h_nd\nu
    \uparrow
    \inf_{\nu:\Wass_p(\nu,\mu)\le\delta}\int hd\nu.
\end{equation}
\end{proposition}

\begin{proof}
The dual representation \eqref{eqn:bm_metric_dual} follows from the seminal work on metric-space optimal-transport duality by \citet{blanchet2019modelrisk}; it is obtained by applying their supremum version to $-h$. Since $h\geq 0$ and $\int h\,d\mu<\infty$, the right-hand side is well defined as $h^\lambda$ is measureable (see the comment after Theorem 1 in \citet{blanchet2019modelrisk}) and $0\leq h^\lambda\leq h$.

To show \eqref{eqn:robust_monotone_continuity}, let
\[
    L:=\lim_{n\ra\infty} \inf_{\nu:\Wass_p(\nu,\mu)\le\delta}\int h_nd\nu, \quad\text{and}\quad U = \inf_{\nu:\Wass_p(\nu,\mu)\le\delta}\int hd\nu. 
\]
The first limit exists due to monotonicity. Note that monotonicity also gives $L\leq U$. So it remains to show that $L\ge U$. Choose $\nu_n$ in the Wasserstein ball such that
\[
    \int h_nd\nu_n\le \inf_{\nu:\Wass_p(\nu,\mu)\le\delta}\int h_nd\nu+n^{-1}.
\]
The Wasserstein ball of measures on a proper Polish space is weakly compact. Therefore, after passing to a subsequence which we also denote by $\set{\nu_n:n\ge 1}$, $\nu_n\Rightarrow\nu^*$ for some feasible $\nu^*$.  

Next, for each fixed $k$, the lower semicontinuity, Portmanteau theorem, monotonicity, and the choice of $\nu_n$ imply
\[
    \int h_kd\nu^*
    \le \liminf_{n\to\infty}\int h_kd\nu_n
    \le \liminf_{n\to\infty}\int h_nd\nu_n
    \le L.
\]
Since $\nu^*$ is feasible, letting $k\to\infty$ and using monotone convergence yields $U \leq \int hd\nu^*\le L$. This completes the proof.
\end{proof}

In addition to duality and approximation from below, our main results leverage the following radius-sensitivity estimate to control the local growth of Wasserstein robust values. 

\begin{lemma}\label{lemma:radius_sensitivity}
    Let $(E,\cE)$ denote either $(\X,\cX)$ or $(\W,\cW)$. Let $h:E\to[0,+\infty]$ be lower semicontinuous with $\int hd\mu\le H$ and $\mu\in\cP_p(\cE)$. Then
\begin{equation}
\label{eqn:radius_sensitivity}
    0\le
     \inf_{\nu:\Wass_p(\nu,\mu)\le \delta}\int hd\nu- \inf_{\nu:\Wass_p(\nu,\mu)\le \delta+\rho}\int hd\nu
    \le
    p2^{p-1}H\frac{\rho}{\delta},
\end{equation}
for all $\delta>0$ and $\rho\ge0$. 
\end{lemma}

\begin{proof}
    
To show \eqref{eqn:radius_sensitivity}, we consider the dual in Proposition \ref{prop:wasserstein_lower_duality}
\[    \inf_{\nu:\Wass_p(\nu,\mu)\le\delta}
    \int hd\nu
    =
    \sup_{\lambda\ge0}
    \left\{\int h^\lambda(z)\mu(dz)-\lambda\delta^p\right\}.
\]
The dual multipliers may be restricted to
$0\le\lambda\le H/\delta^p$, because $0\le h^\lambda\le h$.  

We treat $\rho \in[0,\delta]$ and $\rho > \delta$ separately. If $0\le\rho\le\delta$, then by the mean value theorem $(\delta+\rho)^p-\delta^p= p(\xi)^{p-1}\rho\leq p2^{p-1} \delta^{p-1}\rho,$ where $\xi\in[\delta,\delta+\rho]$. Therefore, we have
\[
    \inf_{\nu:\Wass_p(\nu,\mu)\le \delta}\int hd\nu- \inf_{\nu:\Wass_p(\nu,\mu)\le \delta+\rho}\int hd\nu
    \le
    \sup_{0\le\lambda\le H/\delta^p}
    \lambda\{(\delta+\rho)^p-\delta^p\}
    \le
    p2^{p-1}H\frac{\rho}{\delta},
\]
where the second inequality follows from $\sup f-\sup g\leq \sup (f-g)$. On the other hand, if $\rho>\delta$,
$$\inf_{\nu:\Wass_p(\nu,\mu)\le \delta}\int hd\nu- \inf_{\nu:\Wass_p(\nu,\mu)\le \delta+\rho}\int hd\nu
    \le\inf_{\nu:\Wass_p(\nu,\mu)\le \delta}\int hd\nu\le H\le p2^{p-1}H\frac{\rho}{\delta}.$$
Hence \eqref{eqn:radius_sensitivity} holds. 
\end{proof}

\subsection{Weighted Spaces and Lipschitz Continuity}

Since we consider possibly unbounded spaces, we study Lipschitz regularity under $\kappa$-weighted norms, which allows for super-linear growth. Let $\kappa:\X\to[1,\infty)$ be a locally bounded weight; i.e. $\sup_{x\in K} \kappa (x) < \infty$ for every compact $K\subset\X$. It encodes the growth of rewards and value functions and the drift of the underlying controlled Markov chain. Common choices of weights include $\kappa(x) = 1+d_\X(x,x_0)^\rho$ for some $\rho\ge 1$, or $\kappa(x) = 1+\log (1+d_\X(x,x_0))$.  

For $v:\X\ra\R$, define the weighted sup-norm and induced function space: 
\[
    \norm{v}_\kappa:=\sup_{x\in\X}\frac{\abs{v(x)}}{\kappa(x)},
    \qquad
    \Bkap:=\set{v\in m\set{\cX\ra\cB(\R)}:\norm{v}_\kappa<\infty},
    \qquad
    \Bkap^+:=\set{v\in\Bkap:v\ge0}, 
\]
where $\cB_\kappa$ is a Banach space. 

The weighted Lipschitz seminorm is defined as
\begin{equation}
\label{eqn:weighted_lip_holder}
    [v]_{\kappa}
    :=
    \sup_{x\ne y}
    \frac{\abs{v(x)-v(y)}}{(\kappa(x)+\kappa(y))d_\X(x,y)}.
\end{equation}
We note that since $\kappa$ is locally bounded, a finite weighted Lipschitz seminorm implies ordinary continuity on compacts.

\section{Kernel Wasserstein Ambiguity}
\label{section:kernel}

We first analyze the kernel-robust operator \eqref{eqn:kernel_operator}.  The reward is allowed to depend on the current state, the action, and the next state.  

\begin{assumption}
\label{assump:kernel}
Assume the following conditions hold.
\begin{enumerate}[label=\textup{(\roman*)},leftmargin=*]
    \item The reward $r:\X\times\A\times\X\to\R_+$ is continuous. 
    There exists $L_r<\infty$ such that, for all $x,x',y\in\X$ and
    $a\in\A$,
    \begin{equation}
    \label{eqn:r_lipschitz_in_x}
        |r(x,a,y)-r(x',a,y)|
        \le
        L_r(\kappa(x)+\kappa(x'))d_\X(x,x').
    \end{equation}
    Moreover, there exists $R_\kappa<\infty$ such that, for all $(x,a)\in\X\times\A$,
    \begin{equation}
    \label{eqn:kernel_rwd_bd}
        \int_\X r(x,a,y)P_0(dy\mid x,a)
        \le
        R_\kappa\kappa(x).
    \end{equation}

    \item The nominal kernel $P_0$ is a Borel stochastic kernel with
    $P_0(\cdot\mid x,a)\in\cP_p(\cX)$ for all $(x,a)\in\X\times\A$.  There exists
    $D_\kappa<\infty$ and $\gamma D_\kappa<1$ such that
    \begin{equation}
    \label{eqn:kernel_nominal_drift}
        \int_\X \kappa(y)P_0(dy\mid x,a)
        \le
        D_\kappa\kappa(x)
        .
    \end{equation}

    \item There exists $L_P<\infty$ such that, for all $x,x'\in\X$ and $a\in\A$,
    \begin{equation}
    \label{eqn:kernel_wp_lipschitz}
        \Wass_p\!\left(P_0(\cdot\mid x,a),P_0(\cdot\mid x',a)\right)
        \le
        L_Pd_\X(x,x').
    \end{equation}
\end{enumerate}
\end{assumption}

\begin{theorem}\label{thm:kernel_main}
Suppose Assumption \ref{assump:kernel} holds.  Fix $\delta> 0$, let $V_{0,\delta}\equiv0$ and $V_{n+1,\delta}:=\cT_\delta V_{n,\delta}$ for all $n\ge 0$.  Then $\set{V_{n,\delta}:n\ge0}$ is pointwise non-decreasing and
\begin{equation}
\label{eqn:kernel_barV}
    \norm{V_{n,\delta}}_\kappa
    \le
    \bar V_\kappa
    :=
    \frac{R_\kappa}{1-\gamma D_\kappa}
\end{equation} for all $n\ge0$. Moreover, the pointwise limit $V_{n,\delta}\ua V_\delta\in \Bkap^+$
exists and is the minimal non-negative solution of $V_\delta=\cT_\delta V_\delta$ in $\Bkap^+$.  Furthermore, 
\begin{equation}
\label{eqn:kernel_main_lip}
    [V_\delta]_{\kappa}
    \le
    L_r+
    p2^{p-1}L_P
    \frac{R_\kappa+\gamma D_\kappa\bar V_\kappa}{\delta}.
\end{equation}
In particular, $[V_\delta]_{\kappa}=O(\delta^{-1})$ as $\delta\downarrow0$.
\end{theorem}

\begin{proof}

We begin by establishing a one-step estimate that illustrates the Lipschitz regularization effect of the robust Bellman operator. This estimate not only underlies the Lipschitz continuity of the fixed point, but also ensures that $\set{V_{n,\delta}:n\geq 0}$ are continuous, and hence measurable.

\paragraph{One-step Lipschitz regularization.}
We show that for continuous $u\in\Bkap^+$,
\begin{equation}
\label{eqn:kernel_one_step_inline}
    [\cT_\delta u]_{\kappa}
    \le
    L_r+
    p2^{p-1}L_P
    \frac{R_\kappa+\gamma D_\kappa\|u\|_\kappa}{\delta}.
\end{equation}

Fix a continuous $u\in\Bkap^+$ and define
\[
    I_\delta(x,a)
    :=
    \inf_{\nu\in\cP_\delta(x,a)}
    \int_\X [r(x,a,y)+\gamma u(y)]\nu(dy).
\]
Set $B:=R_\kappa+\gamma D_\kappa\|u\|_\kappa$, then by \eqref{eqn:kernel_rwd_bd} and \eqref{eqn:kernel_nominal_drift}, for every $(x,a)\in\X\times\A$,
\begin{equation}
\label{eqn:center_expectation_bound_kernel}
    \int_\X [r(x,a,y)+\gamma u(y)]P_0(dy| x,a)
    \le
    B\kappa(x).
\end{equation}

To prove \eqref{eqn:kernel_one_step_inline}, we enlarge the ambiguity ball and apply the radius-sensitivity estimate in Lemma \ref{lemma:radius_sensitivity}. Let $\rho:=\Wass_p(P_0(\cdot| x,a),P_0(\cdot|x',a)).$
By triangle inequality, $P_\delta(x,a)\subset P_{\delta+\rho}(x',a)$.  Moreover, by
\eqref{eqn:r_lipschitz_in_x},
\[
    r(x',a,y)+\gamma u(y)
    \le
    r(x,a,y)+\gamma u(y)
    +
    L_r(\kappa(x)+\kappa(x'))d_\X(x,x').
\]
It follows that
\[
    I_{\delta+\rho}(x',a)
    \le \inf_{\nu\in\cP_\delta(x,a)}
    \int_\X [r(x',a,y)+\gamma u(y)]\nu(dy)\leq 
    I_\delta(x,a)
    +
    L_r(\kappa(x)+\kappa(x'))d_\X(x,x').
\]
Rearrange and add $I_\delta(x',a)$ on both sides, 
\[
\begin{aligned}
    I_\delta(x',a)-I_\delta(x,a) 
    &\le
    I_\delta(x',a) - I_{\delta+\rho}(x',a)
    +
    L_r(\kappa(x)+\kappa(x'))d_\X(x,x')  \\
    &\stackrel{(i)}{\le}
    p2^{p-1}B\kappa(x')\frac{\rho}{\delta}
    +
    L_r(\kappa(x)+\kappa(x'))d_\X(x,x')  \\
    &\le
    \left[
        L_r+
        p2^{p-1}L_P\frac{B}{\delta}
    \right]
    (\kappa(x)+\kappa(x'))d_\X(x,x')
\end{aligned}
\]
 where $(i)$ applies \eqref{eqn:radius_sensitivity} with center $\mu=P_0(\cdot|x',a)$ and $h(y) = r(x',a,y)+\gamma u(y)$ with $H= B\kappa(x')$ given by \eqref{eqn:center_expectation_bound_kernel}, and the last inequality uses \eqref{eqn:kernel_wp_lipschitz}, which gives $\rho = \Wass_p(P_0(\cdot| x,a),P_0(\cdot|x',a))\le L_Pd_\X(x,x')$. 
 
Then, exchanging $x$ and $x'$ gives the same estimate with the opposite sign. Therefore
\[
    |I_\delta(x',a)-I_\delta(x,a)|
    \le
    \left[
        L_r+
        p2^{p-1}L_P\frac{R_\kappa+\gamma D_\kappa\|u\|_\kappa}{\delta}
    \right]
    (\kappa(x)+\kappa(x'))d_\X(x,x').
\]
Since $|\sup_a f- \sup_a g|\leq \sup_a |f-g|$, this implies \eqref{eqn:kernel_one_step_inline}.

\paragraph{Monotonicity, boundedness, and regularity.}
We first show that the operator $\cT_\delta$ is monotone on $\Bkap^+$; i.e. $\cT_\delta u\le \cT_\delta v $ if $u\le v$.  Indeed, if $u\le v$, then
for every feasible transition law $\nu\in\cP_\delta(x,a)$, $\int_\X [r(x,a,y)+\gamma u(y)]\nu(dy)
    \le
    \int_\X [r(x,a,y)+\gamma v(y)]\nu(dy)$.

Moreover, for $u\in\Bkap^+$, the center law $P_0(\cdot\mid x,a)$ is feasible for
$\cP_\delta(x,a)$.  Thus, using \eqref{eqn:kernel_rwd_bd} and
\eqref{eqn:kernel_nominal_drift},
\begin{equation}\label{eqn:kernel_Tu_ub}
\begin{aligned}
    \cT_\delta u(x)
    \le
    \sup_{a\in\A}
    \int_\X [r(x,a,y)+\gamma u(y)]P_0(dy|x,a)
    \le
    \bigl(R_\kappa+\gamma D_\kappa\|u\|_\kappa\bigr)\kappa(x).
\end{aligned}
\end{equation}
Hence, $\|\cT_\delta u\|_\kappa
    \le
    R_\kappa+\gamma D_\kappa\|u\|_\kappa.$

We then use these properties and \eqref{eqn:kernel_one_step_inline} inductively to show that for all $n\ge 1$, $V_{n,\delta}\in \Bkap^+$, $V_{n,\delta} \ge V_{n-1,\delta}$, $\norm{V_{n,\delta}}_\kappa \leq \bar V_\kappa$ and \[
    [V_{n,\delta}]_{\kappa}
    \le
    L_r+
    p2^{p-1}L_P
    \frac{R_\kappa+\gamma D_\kappa\bar V_\kappa}{\delta}.
\]

For the base case, since $r\ge0$, $V_{1,\delta}\ge 0 =V_{0,\delta}$. Moreover, with $u = 0$ continuous, \eqref{eqn:kernel_one_step_inline} applies. We have
\[
[V_{1,\delta}]_{\kappa} = [\cT_{\delta}V_{0,\delta}]_{\kappa}\leq L_r+
    p2^{p-1}L_P
    \frac{R_\kappa}{\delta}.
\]
In particular, $V_{1,\delta}$ is continuous and belongs to $\Bkap^+$. For the induction step, assume the induction hypothesis for $V_{k,\delta}$. Since $V_{k,\delta}\in\Bkap^+$, \eqref{eqn:kernel_Tu_ub} implies that 
\begin{equation}\label{eqn:kernel_induction_kappa_norm_est}
    \norm{V_{k+1,\delta}}_\kappa = \|\cT_\delta V_{k,\delta}\|_\kappa
\le   R_\kappa+\gamma D_\kappa\|V_{k,\delta}\|_\kappa\leq \bar V_\kappa.
\end{equation}
By the monotonicity, $V_{k+1,\delta} = \cT_\delta V_{k,\delta} \geq \cT_{\delta}V_{k-1,\delta} = V_{k,\delta}$. Next, with continuous $u=V_{k,\delta}\in\cB_\kappa^+$, we apply \eqref{eqn:kernel_one_step_inline} to get
\[
[V_{k+1,\delta}]_{\kappa} = [\cT_{\delta}V_{k,\delta}]_{\kappa}\le L_r+
    p2^{p-1}L_P
    \frac{R_\kappa+\gamma D_\kappa\bar V_\kappa}{\delta}. 
\]
In particular, $V_{k+1,\delta}$ is continuous and hence measurable. This, together with \eqref{eqn:kernel_induction_kappa_norm_est} and monotonicity, implies that $V_{k+1,\delta}\in\Bkap^+$.

\paragraph{Limiting fixed-point and minimality.} Since $\set{V_{n,\delta}:n\ge 0}$ are non-decreasing and uniformly bounded in $\norm{\cd}_\kappa$, the pointwise limit $V_\delta$ exists and belongs to $\Bkap^+$. Moreover, for all $x\ne x'$, 
\[
    \frac{|V_\delta(x)-V_\delta(x')|}
    {(\kappa(x)+\kappa(x'))d_\X(x,x')}
    =
    \lim_{n\to\infty}
    \frac{|V_{n,\delta}(x)-V_{n,\delta}(x')|}
    {(\kappa(x)+\kappa(x'))d_\X(x,x')}
    \leq L_r+
    p2^{p-1}L_P
    \frac{R_\kappa+\gamma D_\kappa\bar V_\kappa}{\delta}.
\]
This implies \eqref{eqn:kernel_main_lip}.

The one-step estimate shows that $V_{n,\delta}$ is continuous for every $n\ge1$.
Fix $(x,a)$ and define
\[
    h_n(y):=r(x,a,y)+\gamma V_{n,\delta}(y),
    \qquad
    h(y):=r(x,a,y)+\gamma V_\delta(y).
\]
Then $h_n\uparrow h$, each $h_n$ and $h$ are non-negative and continuous, and
\[
    \int_\X h(y)P_0(dy\mid x,a)
    \le
    (R_\kappa+\gamma D_\kappa\bar V_\kappa)\kappa(x)<\infty.
\]
So, Proposition \ref{prop:wasserstein_lower_duality} applies with $\mu = P_0(\cd|x,a)$. We conclude that
\begin{align*}
    V_{\delta }(x) 
    &=\lim_{n\ra\infty}\cT_\delta V_{n,\delta}(x) = \lim_{n\ra\infty}\sup_{a\in\A}\inf_{\nu\in\cP_\delta(x,a)}\int h_nd\nu\\
    &\stackrel{(i)}{=}\sup_{n\geq 0}\sup_{a\in\A}
    \inf_{\nu\in\cP_\delta(x,a)}\int h_nd\nu\\
    &\stackrel{(ii)}{=}\sup_{a\in\A}
    \lim_{n\ra\infty}\inf_{\nu\in\cP_\delta(x,a)}\int h_nd\nu\\
    &= \cT_{\delta}V_{\delta}(x)
\end{align*}
where $(i)$ and $(ii)$ follow from the non-decreasing property of $h_n$, and the last equality follows from Proposition \ref{prop:wasserstein_lower_duality}.

Finally, we show that $V_\delta$ is minimal. If $U\in\Bkap^+$ is another non-negative fixed point of $\cT_\delta$, then
$V_{0,\delta}=0\le U$.  Monotonicity gives $V_{n,\delta}\le U$ for all $n$, and
therefore $V_\delta\le U$. \end{proof}

\section{Exogenous Noise Wasserstein Ambiguity}
\label{section:caa_cau}

This section studies the robust Bellman operators \eqref{eqn:caa_operator} and
\eqref{eqn:cau_operator}.  The argument mirrors the kernel case.  The nominal
growth assumptions control weighted boundedness, while Lipschitz regularity of the robust value comes from a Lipschitz property of the next-state law as a
function of the current state.

\begin{assumption}[Primitive conditions for exogenous-noise ambiguity]
\label{assump:noise}
Assume the following conditions hold.
\begin{enumerate}[label=\textup{(\roman*)},leftmargin=*]
    \item The functions $r:\X\times\A\times\W\to\R_+$ and
    $f:\X\times\A\times\W\to\X$ are continuous.  There exists $L_r<\infty$ such
    that, for all $x,y\in\X$, $a\in\A$, and $w,w'\in\W$,
    \begin{equation}
    \label{eqn:noise_reward_lip}
        |r(x,a,w)-r(y,a,w')|
        \le
        L_r(\kappa(x)+\kappa(y))
        \bigl(d_\X(x,y)+d_\W(w,w')\bigr).
    \end{equation}

    \item There exist $R_\kappa,D_\kappa<\infty$ with $\gamma D_\kappa<1$ such
    that, for all $(x,a)\in\X\times\A$,
    \begin{equation}
    \label{eqn:noise_rwd_drift_bd}
        \int_\W r(x,a,w)\psi_0(dw)
        \le R_\kappa\kappa(x),
        \qquad
        \int_\W \kappa(f(x,a,w))\psi_0(dw)
        \le D_\kappa\kappa(x).
    \end{equation}
\end{enumerate}
\end{assumption}

Assumption \ref{assump:noise} mirrors parts (i) and (ii) in Assumption \ref{assump:kernel}. To enable the same mechanism that leads to Theorem \ref{thm:kernel_main}, we want to consider an analogue of the kernel condition
\[
    \Wass_p(P_0(\cdot\mid x,a),P_0(\cdot\mid y,a))
    \le L_Pd_\X(x,y)
\]
for the Wasserstein robust exogenous-noise model. 

In the kernel model, the ambiguity set is a Wasserstein ball directly over
next-state laws.  In the exogenous-noise model, the ambiguity set is instead a
Wasserstein ball over noise laws, while the induced next-state ambiguous laws, given by the pushed-forward measures, are usually not a distribution ball in any distance. Thus, a Lipschitz condition imposed only on the center $\psi_0$ could fail to control the
induced law for the next state. Instead, we impose the following condition in Assumption \ref{assump:noise_law_lip}. We will discuss sufficient and easy-to-verify conditions in Proposition \ref{prop:noise_sufficient_cond}.

\begin{assumption}
\label{assump:noise_law_lip}
There exists $L_P<\infty$ such that the following holds.
\begin{enumerate}[label=\textup{(\roman*)},leftmargin=*]
    \item For the CAA case, assume that for each $\delta>0$, $x,y\in\X$, $a\in\A$, and $\psi\in\cU_\delta$, there exists $\psi'\in\cP_p(\cW)$ with $\Wass_p(\psi,\psi')\le L_Pd_\X(x,y)$ such that
    \[\law_{W\sim\psi}(f(x,a,W))
        = \law_{W\sim\psi'}(f(y,a,W)).
    \]

    \item  For the CAU case, assume that for each $\delta>0$, $x,y\in\X$, $\phi\in\cP(\cA)$, and $\psi\in\cU_\delta$, there exists $\psi'\in\cP_p(\cW)$ with $\Wass_p(\psi,\psi')\le L_Pd_\X(x,y)$ such that
    \[
        \law_{(A,W)\sim\phi\times\psi}(f(x,A,W))
        = \law_{(A,W)\sim\phi\times\psi'}(f(y,A,W)).
    \]
\end{enumerate}
\end{assumption}
\begin{remark}\label{rmk:law_matching}
   We observe that the next-state law-matching requirements in Assumption~\ref{assump:noise_law_lip} can be equivalently written using pushforward measures: \begin{equation}\label{eqn:noise_next_tstate_law_match}f(x,\cdot,\cdot)_\sharp(\phi\times\psi)=f(y,\cdot,\cdot)_\sharp(\phi\times\psi').
   \end{equation} Here, $\phi\in\set{\delta_a:a\in\A}$ in the CAA case, and $\phi\in\cP(\cA)$ in the CAU case. In particular, this shows that the CAU condition is sufficient for the CAA case. But the converse is not true in general.

   We also note that the exact law matching condition in \eqref{eqn:noise_next_tstate_law_match} can be further relaxed to allow deviations in a weighted optimal transport distance that are Lipschitz in the state; see Theorem~\ref{thm:noise_relaxed}. While this flexibility can be useful, especially in unbounded state-space settings, the resulting presentation is more involved. We therefore first present the main result under the cleaner setup in Assumption \ref{assump:noise_law_lip}.
\end{remark}

\begin{theorem}
\label{thm:caa_cau_main}
Suppose Assumptions \ref{assump:noise} and \ref{assump:noise_law_lip} hold.  Fix
$\delta>0$ and $\mrm{M}\in\{\mrm{CAA},\mrm{CAU}\}$.  Let
$V_{0,\delta}^{\mrm{M}}\equiv0$ and
$V_{n+1,\delta}^{\mrm{M}}:=\cT_\delta^{\mrm{M}}V_{n,\delta}^{\mrm{M}}$
for all $n\ge0$.  Then $\{V_{n,\delta}^{\mrm{M}}:n\ge0\}$ is pointwise non-decreasing and
\begin{equation}
\label{eqn:noise_barV}
    \norm{V_{n,\delta}^{\mrm{M}}}_\kappa
    \le
    \bar V_\kappa
    :=
    \frac{R_\kappa}{1-\gamma D_\kappa}
\end{equation}
for all $n\ge0$. Moreover, the pointwise limit
$V_{n,\delta}^{\mrm{M}}\ua V_\delta^{\mrm{M}}\in\Bkap^+$
exists and is the minimal non-negative solution of
$V_\delta^{\mrm{M}}=\cT_\delta^{\mrm{M}}V_\delta^{\mrm{M}}$ in $\Bkap^+$.
Furthermore, 
\begin{equation}
\label{eqn:noise_main_lip}
    [V_\delta^{\mrm{M}}]_{\kappa}
    \le
    L_r(1+L_P)
    +
    p2^{p-1}L_P
    \frac{R_\kappa+\gamma D_\kappa\bar V_\kappa}{\delta}.
\end{equation}
In particular,
$[V_\delta^{\mrm{M}}]_{\kappa}=O(\delta^{-1})$ as $\delta\downarrow0$.
\end{theorem}

\begin{proof}
To simplify notation, in the CAA case, write $\Phi_{\mrm{CAA}}:=\{\delta_a:a\in\A\},$
and in the CAU case, write $\Phi_{\mrm{CAU}}:=\cP(\cA).$

The proof mirrors that of the kernel robust case. We begin by establishing the one-step Lipschitz estimate as in the kernel robust case. The main difference is that the state-wise robust value comparison is performed by changing the noise law and using the next-state law Lipschitz condition in
Assumption \ref{assump:noise_law_lip}.

\paragraph{One-step Lipschitz regularization.}
We show that for continuous $u\in\Bkap^+$,
\begin{equation}
\label{eqn:noise_one_step_inline}
    [\cT_\delta^{\mrm{M}}u]_{\kappa}
    \le
    L_r(1+L_P)
    +
    p2^{p-1}L_P
    \frac{R_\kappa+\gamma D_\kappa\norm{u}_\kappa}{\delta}.
\end{equation}

We begin by establishing some preliminary bounds. Fix a continuous $u\in\Bkap^+$.  
For $\phi\in\Phi_{\mrm{M}}$, define
\[
    h_x^\phi(w)
    :=
    \int_\A
    \sqbk{r(x,a,w)+\gamma u(f(x,a,w))}
    \phi(da). 
\]
We note that $h_x^\phi$ is lower semicontinuous by Fatou's lemma. Indeed, for $w_n\ra w$, 
\begin{equation}\label{eqn:tu_lsc}h_x^\phi(w)
    =
    \int_\A
    \lim _{n\ra\infty}\sqbk{r(x,a,w_n)+\gamma u(f(x,a,w_n))}
    \phi(da)\leq \liminf_{n\ra\infty} h_x^\phi(w_n).
\end{equation}
Also, we note that since $\A$ could be unbounded, $h_x^\phi$ could be $+\infty$ for some $w$.

As in the kernel robust case, define
\[
    I_\delta(x,\phi)
    :=
    \inf_{\psi\in\cU_\delta}
    \int_\W h_x^\phi(w)\psi(dw).
\]
Then by construction, $\cT_\delta^{\mrm{M}}u(x)=\sup_{\phi\in\Phi_{\mrm{M}}} I_\delta(x,\phi).$
Moreover, set $B:=R_\kappa+\gamma D_\kappa\norm{u}_\kappa$. By the reward and drift bound \eqref{eqn:noise_rwd_drift_bd} in Assumption \ref{assump:noise}, for every
$x\in\X$ and every $\phi\in\Phi_{\mrm{M}}$,
\begin{equation}
\label{eqn:noise_center_expectation_bound}
    \int_\W h_x^\phi(w)\psi_0(dw)
    \le
    B\kappa(x).
\end{equation}
This also implies that $I_\delta(x,\phi) < \infty$.

We now proceed to show \eqref{eqn:noise_one_step_inline}.  Fix $x,y\in\X$ and $\phi\in\Phi_{\mrm{M}}$. By the definition of $I_\delta$, for every $\epsilon > 0$, we can find $\psi\in\cU_\delta$ such that
\begin{equation}\label{eqn:tu_psi_choice}
    \int_\W h_x^\phi(w)\psi(dw)
    \le
    I_\delta(x,\phi)+\epsilon < \infty.
\end{equation}

By Assumption \ref{assump:noise_law_lip}, there exists
$\psi'\in\cP_p(\cW)$ such that
\begin{equation}\label{eqn:tu_psiprime_choice}
    \Wass_p(\psi,\psi')\le L_Pd_\X(x,y)=:\rho
\end{equation}
and \eqref{eqn:noise_next_tstate_law_match} holds for $\phi\in\Phi_\mrm{M}$. Then, $\Wass_p(\psi',\psi_0) \leq \rho + \delta$ and hence $\psi'\in\cU_{\delta+\rho}.$
Moreover, for $u\in\Bkap^+$,
\begin{equation}
\label{eqn:noise_continuation_identity}
    \int_\W\int_\A u(f(y,a,w))\phi(da)\psi'(dw)
    =
    \int_\W\int_\A u(f(x,a,w))\phi(da)\psi(dw).
\end{equation}

On the other hand, for the reward, we first note that by non-negativity and \eqref{eqn:tu_psi_choice}
\[
0\leq \int_\W\int_\A r(x,a,w)\phi(da)\psi(dw) \leq \int_\W h_x^\phi(w)\psi(dw)  < \infty
\]
Then, for any $\epsilon'>0$, let $\pi\in\Pi(\psi,\psi')$ be an $\epsilon'$-optimal coupling for the $L^1$ Wasserstein distance; i.e.
\[\int_{\W\times\W}d_\W(w,w')\pi(dw,dw') \leq \Wass_1(\psi,\psi') + \epsilon'.
\] By the finiteness of the first reward integral and the 
Lipschitz condition in Assumption \ref{assump:noise},
\[
\begin{aligned}
&\abs{
    \int_\W\int_\A r(x,a,w)\phi(da)\psi(dw)
    -
    \int_\W\int_\A r(y,a,w)\phi(da)\psi'(dw)}\\
&= \abs{\int_\W\int_\A [r(x,a,w)- r(y,a,w')] \phi(da)\pi(dw,dw')}\\
&\stackrel{(i)}{\le}
    L_r(\kappa(x)+\kappa(y))
    \left(d_\X(x,y)+\int_{\W\times\W}d_\W(w,w')\pi(dw,dw')\right)\\
    &\leq L_r(\kappa(x)+\kappa(y))(d_\X(x,y)+ \Wass_1(\psi,\psi') + \epsilon')\\
    &\leq L_r(\kappa(x)+\kappa(y))(d_\X(x,y)+ \Wass_p(\psi,\psi') + \epsilon')\\
\end{aligned}
\]
where $(i)$ follows from \eqref{eqn:noise_reward_lip} and the last line follows from Jensen's inequality. Since $\epsilon'$ can be arbitrarily small and recalling the choice $\psi'$ in \eqref{eqn:tu_psiprime_choice}, we conclude that
\[
\begin{aligned}
&\left|
    \int_\W\int_\A r(x,a,w)\phi(da)\psi(dw)
    -
    \int_\W\int_\A r(y,a,w)\phi(da)\psi'(dw)
\right|                                      \\
&\qquad\le
    L_r(1+L_P)(\kappa(x)+\kappa(y))d_\X(x,y).
\end{aligned}
\]
Together with \eqref{eqn:noise_continuation_identity}, this gives
\begin{equation}\label{eqn:tu_h_x_y_bd}
    \int_\W h_y^\phi(w)\psi'(dw)
    \le
    \int_\W h_x^\phi(w)\psi(dw)
    +
    L_r(1+L_P)(\kappa(x)+\kappa(y))d_\X(x,y).
\end{equation}
Since $\psi'\in\cU_{\delta+\rho}$ and $\psi$, as specified in \eqref{eqn:tu_psi_choice}, is $\epsilon$-optimal for
$I_\delta(x,\phi)$ and $\epsilon$ can be arbitrarily small,  by letting $\epsilon\downarrow0$ we obtain
\begin{equation}
\label{eqn:noise_state_radius_comparison}
    I_{\delta+\rho}(y,\phi)
    \le
    I_\delta(x,\phi)
    +
    L_r(1+L_P)(\kappa(x)+\kappa(y))d_\X(x,y).
\end{equation}

On the other hand, applying the radius-sensitivity estimate in Lemma \ref{lemma:radius_sensitivity} with center $\mu = \psi_0\in\cP_p(\cW)$, lower semicontinuous integrand
$h_y^\phi :\W\ra [0,+\infty]$, and $H=B\kappa(y)$ from
\eqref{eqn:noise_center_expectation_bound}, we get
\[
    0\le
    I_\delta(y,\phi)-I_{\delta+\rho}(y,\phi)
    \le
    p2^{p-1}B\kappa(y)\frac{\rho}{\delta}.
\]
Combining this with \eqref{eqn:noise_state_radius_comparison} and recalling the definition $\rho=L_Pd_\X(x,y)$  in \eqref{eqn:tu_psiprime_choice}, we have
\[
\begin{aligned}
    I_\delta(y,\phi)-I_\delta(x,\phi)&=I_\delta(y,\phi)\pm I_{\delta+\rho}(y,\phi) - I_\delta(x,\phi)\\
    &\le
    \left[
        L_r(1+L_P)
        +
        p2^{p-1}L_P\frac{B}{\delta}
    \right]
    (\kappa(x)+\kappa(y))d_\X(x,y).
\end{aligned}
\]
Exchanging $x$ and $y$ gives the same estimate with the opposite sign.  Hence, recalling the definition of $B$ and taking sup over $\phi\in\Phi_{\mrm{M}}$ yields
\begin{align*}
    \abs{\cT_{\delta}^{\mrm{M}} u(x)-\cT_{\delta}^{\mrm{M}} u(y)}&=\abs{\sup_{\phi\in\Phi_{\mrm{M}}}I_\delta(x,\phi)-\sup_{\phi\in\Phi_{\mrm{M}}}I_\delta(y,\phi)}
    \\
    &\leq \sup_{\phi\in\Phi_{\mrm{M}}}\abs{I_\delta(x,\phi)-I_\delta(y,\phi)}\\
    &\le
    \left[
        L_r(1+L_P)
        +
        p2^{p-1}L_P
        \frac{R_\kappa+\gamma D_\kappa\norm{u}_\kappa}{\delta}
    \right]
    (\kappa(x)+\kappa(y))d_\X(x,y),
\end{align*}
for all $x,y\in\X$. This implies  \eqref{eqn:noise_one_step_inline}.

\paragraph{Monotonicity, boundedness, and regularity.}
We note that the monotonicity of $\cT_\delta^{\mrm{M}}$ on $\Bkap^+$ follows from the same argument as in the kernel robust case. Therefore, $\{V_{n,\delta}^{\mrm{M}}:n\ge 0\}$ is pointwise increasing. 

Moreover, for $u\in\Bkap^+$, the center law $\psi_0$ belongs to $\cU_\delta$.
Thus, using the nominal growth and drift bounds in Assumption \ref{assump:noise},
\begin{equation}
\label{eqn:noise_Tu_ub}
\begin{aligned}
    \cT_\delta^{\mrm{M}}u(x)
    &\le
    \sup_{\phi\in\Phi_{\mrm{M}}}
    \int_\W\int_\A
    \sqbk{r(x,a,w)+\gamma u(f(x,a,w))}
    \phi(da)\psi_0(dw)                         \\
    &\le
    \bigl(R_\kappa+\gamma D_\kappa\norm{u}_\kappa\bigr)\kappa(x).
\end{aligned}
\end{equation}
Hence $\norm{\cT_\delta^{\mrm{M}}u}_\kappa \le R_\kappa+\gamma D_\kappa\norm{u}_\kappa.$
\par Using this estimate, monotonicity, and \eqref{eqn:noise_one_step_inline}
inductively, we obtain, for all $n\ge1$, $V_{n,\delta}^{\mrm{M}}\in\Bkap^+,$ $V_{n,\delta}^{\mrm{M}}\ge V_{n-1,\delta}^{\mrm{M}},$ $\norm{V_{n,\delta}^{\mrm{M}}}_\kappa\le\bar V_\kappa,$
and
\[
    [V_{n,\delta}^{\mrm{M}}]_{\kappa}
    \le
    L_r(1+L_P)
    +
    p2^{p-1}L_P
    \frac{R_\kappa+\gamma D_\kappa\bar V_\kappa}{\delta}.
\]
The base case follows from $u=0$, which is continuous.  The induction step
uses that a finite weighted Lipschitz seminorm implies continuity.  This proves
\eqref{eqn:noise_barV}.

Since the sequence is non-decreasing and uniformly bounded in
$\norm{\cdot}_\kappa$, the pointwise limit 
$V_\delta^{\mrm{M}}(x):=\lim_{n\to\infty}V_{n,\delta}^{\mrm{M}}(x)$ exists. Moreover, for every $x\ne y$,
\[
    \frac{|V_\delta^{\mrm{M}}(x)-V_\delta^{\mrm{M}}(y)|}
    {(\kappa(x)+\kappa(y))d_\X(x,y)}
    =
    \lim_{n\to\infty}
    \frac{|V_{n,\delta}^{\mrm{M}}(x)-V_{n,\delta}^{\mrm{M}}(y)|}
    {(\kappa(x)+\kappa(y))d_\X(x,y)}.
\]
Therefore \eqref{eqn:noise_main_lip} follows by taking the supremum over
$x\ne y$. In particular, this also implies that $V_\delta^{\mrm{M}}$ is measurable and hence in $\Bkap^+$. 

\paragraph{Limiting fixed point and minimality.}
Fix $x\in\X$ and $\phi\in\Phi_{\mrm{M}}$.  Define
\[
    h_n^\phi(w)
    :=
    \int_\A
    \sqbk{
        r(x,a,w)+
        \gamma V_{n,\delta}^{\mrm{M}}(f(x,a,w))
    }
    \phi(da),
\]
and define $h^\phi$ analogously with $V_\delta^{\mrm{M}}$ in place of $V_{n,\delta}^{\mrm{M}}$.  

Then $h_n^\phi\uparrow h^\phi$. Moreover, since $r$, $f$, $\{V_{n,\delta}^{\mrm{M}}:n\ge 0\}$, and $V_\delta^{\mrm{M}}$ are continuous, by the same argument in \eqref{eqn:tu_lsc}, each $h_n^\phi$ and $h^\phi$ is non-negative and lower semicontinuous in $w$ for all $n\ge 0$.
Furthermore, by \eqref{eqn:noise_Tu_ub} and
\eqref{eqn:noise_barV},
\[
    \int_\W h^\phi(w)\psi_0(dw)
    \le
    (R_\kappa+\gamma D_\kappa\bar V_\kappa)\kappa(x)<\infty.
\]
Thus, Proposition \ref{prop:wasserstein_lower_duality} applies. Hence, we conclude that for all $x\in\X$, 
\begin{align*}
V_{\delta}^{\mrm{M}}(x) \stackrel{(i)}{=}
    \sup_{\phi\in\Phi_{\mrm{M}}}\sup_{n\ge 0}\inf_{\psi\in\cU_\delta}\int h_n^\phi d\psi\stackrel{(ii)}{=} \sup_{\phi\in\Phi_{\mrm{M}}}\lim_{n\ra\infty}\inf_{\psi\in\cU_\delta}\int h_n^\phi d\psi= \cT_{\delta}^{\mrm{M}} V_{\delta}^{\mrm{M}}(x).
\end{align*}
where $(i)$ follows from the monotonicity of $V_{n,\delta}^{\mrm{M}}(x)$ and $(ii)$ uses Proposition \ref{prop:wasserstein_lower_duality}. 

The minimality follows from the same argument as in the kernel robust case.
\end{proof}
\subsection{A Sufficient Condition on the Transition Dynamics}
As discussed before Assumption~\ref{assump:noise_law_lip}, the law-matching condition there is required to hold for all $\psi\in\cU_\delta$, which can be difficult to verify directly. To address this, we present a sufficient condition directly on $f$ that is easy to check in many standard models arising in operations applications.

\begin{proposition}\label{prop:noise_sufficient_cond}
The following conditions are sufficient for Assumption \ref{assump:noise_law_lip}.

\begin{enumerate}[label=\textup{(\roman*)},leftmargin=*]
    \item For the CAA case, suppose that for every
    $x,y\in\X$, $a\in\A$, and $w\in\W$, there exists $w'\in\W$ such that
    \[
        f(y,a,w')=f(x,a,w),
        \quad\text{and}\quad        d_\W(w,w')\le L_Pd_\X(x,y).
    \]
    Then Assumption \ref{assump:noise_law_lip} part (i) holds.

    \item For the CAU case, suppose that for every
    $x,y\in\X$ and $w\in\W$, there exists $w'\in\W$ such that
    \[
        f(y,a,w')=f(x,a,w),\ \forall a\in\A,\quad\text{and}\quad 
        d_\W(w,w')\le L_Pd_\X(x,y).
    \]
    Then Assumption \ref{assump:noise_law_lip} part (ii) holds.
\end{enumerate}
\end{proposition}
\begin{remark}
    While these conditions are more direct, they are far from necessary for Assumption~\ref{assump:noise_law_lip}, as they require pointwise matching dynamics rather than law matching. The condition in part (i) is usually benign and easy to verify in many stochastic control problems, whereas part (ii) can appear much stronger. Examples of dynamics satisfying part (ii) include $f(x,a,w)=g(h(x,w),a),$ for some functions $g$ and $h$, where $h$ has the following property: for every $x,y\in\X$ and $w\in\W$, there exists $w'$ with $d_{\W}(w,w')\leq L_Pd_\X(x,y)$ such that $h(x,w)=h(y,w')$.
\end{remark}

\begin{proof}
We first state a classical measurable selection result. Recall that $\W$ is Polish. Let $B\in \cW\otimes\cW$ be in the product $\sigma$-field and $\pi_1(B):=\set{w\in\W:(w,w')\in B}$. If $B_w:= \set{w'\in\W: (w,w')\in B}$ is compact for all $w\in\W$, then there exists a Borel measurable function $T:\pi_1(B)\ra \W$ such that $T(w)\in B_w$.  This is a simplified version of \citet[Theorem 5.7.1]{srivastava1998course}.

With this theorem, we prove the CAA and CAU statements separately. 

\paragraph{The CAA Case.} Fix $x,y\in\X$ and $a\in\A$. Consider the set
\[
    B =
    \left\{
        (w,w')\in\W^2:
        f(x,a,w)= f(y,a,w'),\
        d_\W(w,w')\le L_Pd_\X(x,y)
    \right\}.
\]
Note that, under the setup in part (i), $$\pi_1(B):=\set{w\in\W:(w,w')\in B} =\W$$
Since $f$ and $d_\W$ are continuous, for $(w_n,w_n')\in B$ and $(w_n,w_n')\ra (w,w')$, we have that $f(x,a,w)= f(y,a,w')$ and $d_\W(w,w')\le L_Pd_\X(x,y)$. Thus, $B$ is a closed subset of $\W\times \W$, hence Borel.

Moreover, $B_w=\set{w'\in \W:(w,w')\in B} \subset \set{w'\in \W:d_\W(w,w')\leq L_Pd(x,y)}$ is a closed subset of the closed ball in $\W$. Since $\W$ is assumed to be proper, closed balls are compact and thus $B_w$ is compact for any $w\in\W$. Thus, by the measurable section theorem, there exists a measurable $T:\W\ra\W$ such that $T(w)\in B_w$. In particular, $f(x,a,w) = f(y,a,T(w))$. 

Then, for any fixed $\psi\in\cU_\delta$ let $\psi':=T_\sharp\psi = \law_{W\sim\psi}(T(W)).$ Since $T$ is measurable, $\psi'$ is a probability measure on $\cW$. Moreover, by the change of measure formula, 
\begin{align*}
    \int d_\W(w,w_0)^p\psi'(dw) &= \int d_\W(T(w),w_0)^p\psi(dw)\\
    &\stackrel{(i)}\le
    \int [d_\W(w,w_0)+L_Pd_\X(x,y)]^p \psi(dw) \\
    &< \infty,
\end{align*}
where $(i)$ follows from $T(w)\in B_w$ and the triangle inequality, and the last step is from $\psi\in \cU_\delta\subset \cP_p(\cW)$. 
So $\psi'\in\cP_p(\cW)$. 

Moreover, treating $T$ as a transport map, we see that
\[
    \Wass_p(\psi,\psi')
    \le
    \sqbk{\int_\W d_\W(w,T(w))^p\psi(dw)}^{1/p}
    \le
    L_Pd_\X(x,y).
\] 

Finally, since $T(w)$ satisfies $f(x,a,w) = f(y,a,T(w))$ for all $w\in\W$, we have that
\[
    f(y,a,\cdot)_\sharp\psi' = f(y,a,T(\cd))_
\sharp\psi    =
    f(x,a,\cdot)_\sharp\psi.
\]

Together, we have verified all the requirements of $\psi'$ in Assumption \ref{assump:noise_law_lip} (i).

\paragraph{The CAU Case. }Similar to the CAA case, we fix $x,y\in\X$ and consider
\[
    B =
    \left\{
        (w,w')\in\W^2:
        f(x,a,w)=f(y,a,w') \ \forall a\in\A,\;d_\W(w,w')\le L_Pd_\X(x,y)
    \right\}.
\]
Again, we have $\pi_1(B)=\W.$ 

Also, as in the CAA case, $B$ is closed and hence Borel. Indeed, the distance constraint is closed. In addition, for each $a\in\A$, the set $\left\{(w,w')\in\W^2:f(x,a,w)=f(y,a,w')\right\}$
is closed by continuity of $f$. Hence $B$, being the intersection of these
closed sets over $a\in\A$ and the distance constraint, is closed.

Moreover, for each $w\in\W$,
\[
    B_w
    =
    \left\{
        w'\in\W:
        f(x,a,w)=f(y,a,w')\ \forall a\in\A,\
        d_\W(w,w')\le L_Pd_\X(x,y)
    \right\}
\]
is a closed subset of a compact 
ball, hence compact. By the measurable selection theorem, there exists a
Borel measurable map $T:\W\ra\W$ such that $T(w)\in B_w$ for every $w\in\W$.
Thus $f(x,a,w)=f(y,a,T(w))$ 
and $d_\W(w,T(w))\le L_Pd_\X(x,y)$ for every $a\in\A $ and $w\in\W,$

Next, for any fixed $\phi\in\cP(\cA)$ and $\psi\in\cU_\delta$ define $\psi':=T_\sharp\psi.$
Exactly as in the CAA case, the moment estimate gives
$\psi'\in\cP_p(\cW),$ and the coupling induced by the transport map $T$ gives
\[
    \Wass_p(\psi,\psi')
    \le
    \sqbk{\int_\W d_\W(w,T(w))^p\psi(dw)}^{1/p}
    \le
    L_Pd_\X(x,y).
\]

It remains to verify the next-state law matching. For every bounded measurable
$\varphi:\X\ra\R$, 
\[
\begin{aligned}
    \int_{\A\times\W}
        \varphi\bigl(f(y,a,w')\bigr)\,\phi(da)\psi'(dw')
    &=
    \int_{\A\times\W}
        \varphi\bigl(f(y,a,T(w))\bigr)\,\phi(da)\psi(dw)  \\
    &=
    \int_{\A\times\W}
        \varphi\bigl(f(x,a,w)\bigr)\,\phi(da)\psi(dw); 
\end{aligned}
\]
i.e., $f(y,\cdot,\cdot)_\sharp(\phi\otimes\psi')=    f(x,\cdot,\cdot)_\sharp(\phi\otimes\psi).$ This verifies Assumption
\ref{assump:noise_law_lip} (ii).
\end{proof}

\subsection{An Extension of Theorem \ref{thm:caa_cau_main}}
As introduced in Remark~\ref{rmk:law_matching}, the exact law-matching requirement can be relaxed while still preserving Lipschitz regularity. To illustrate this, we first define the following optimal transport cost
$$\Wass_\kappa(\mu,\nu):= \inf_{\pi\in\Pi(\mu,\nu)}\int_{\X\times\X} (\kappa(x) + \kappa(y))d_\X(x,y)\pi(dx,dy).$$

\begin{theorem}\label{thm:noise_relaxed}
Suppose Assumption \ref{assump:noise} holds and $\mrm{M}\in\{\mrm{CAA},\mrm{CAU}\}$. Suppose that there are constants $L_P$ and $L^*$ with $\gamma L^* < 1$ such that the following holds. For any $\delta>0$, $x,y\in\X$, $\phi\in\Phi_{\mrm{M}}$, and $\psi\in\cU_\delta$, there exists $\psi'\in\cP_p(\cW)$ with $\Wass_p(\psi,\psi')\le L_Pd_\X(x,y)$ such that
    \[
        \Wass_\kappa( f(x,\cdot,\cdot)_\sharp(\phi\times\psi) ,f(y,\cdot,\cdot)_\sharp(\phi\times\psi')) \leq L^*(\kappa(x) +\kappa(y)) d_\X(x,y).
    \]
Then the statement of Theorem \ref{thm:caa_cau_main} holds with the estimate in \eqref{eqn:noise_main_lip} replaced by 
\[
    [V_\delta^{\mrm{M}}]_{\kappa}
    \le
    \frac{1}{1-\gamma L^*}\crbk{L_r(1+L_P)
    +
    p2^{p-1}L_P
    \frac{R_\kappa+\gamma D_\kappa\bar V_\kappa}{\delta}}.
\]

\end{theorem}
\begin{proof}
The result follows from the same proof as in Theorem \ref{thm:caa_cau_main}. The main difference is the replacement of \eqref{eqn:noise_one_step_inline} with the following Lipschitz bound. For $u\in\Bkap^+$ with $[u]_\kappa < \infty$,
\begin{equation}\label{eqn:noise_one_step_lip_relaxed}
    [\cT_\delta^{\mrm{M}}u]_{\kappa}
    \le
    L_r(1+L_P)
    +
    p2^{p-1}L_P
    \frac{R_\kappa+\gamma D_\kappa\norm{u}_\kappa}{\delta} + \gamma L^*[u]_\kappa.
\end{equation}

To see this, we follow the same proof with \eqref{eqn:noise_continuation_identity} modified as
\[
\begin{aligned}
    &\abs{\int_\W\int_\A u(f(y,a,w))\phi(da)\psi'(dw) -
    \int_\W\int_\A u(f(x,a,w))\phi(da)\psi(dw)}\\
    &\le
    [u]_\kappa \Wass_\kappa( f(x,\cdot,\cdot)_\sharp(\phi\times\psi) ,f(y,\cdot,\cdot)_\sharp(\phi\times\psi')) \\
    &\le [u]_\kappa L^*(\kappa(x) +\kappa(y)) d_\X(x,y)
\end{aligned}\]
by assumption. Thus, \eqref{eqn:tu_h_x_y_bd} now has an additional contribution from the $\gamma u\circ f$ term; i.e. 
\[
    \int_\W h_y^\phi(w)\psi'(dw)
    \le
    \int_\W h_x^\phi(w)\psi(dw)
    +
    (L_r(1+L_P) + \gamma L^*[u]_\kappa)(\kappa(x)+\kappa(y))d_\X(x,y).\]
Hence, \eqref{eqn:noise_one_step_lip_relaxed} follows from the same argument as in the proof of Theorem~\ref{thm:caa_cau_main}. 

To conclude Theorem~\ref{thm:noise_relaxed}, we note that the Picard iterates have finite weighted Lipschitz seminorms, and hence are continuous. They are also bounded in $\norm{\cdot}_\kappa$ and monotone nondecreasing. By induction, they satisfy the desired uniform Lipschitz bound. The same continuity argument then shows that the monotone limit of the Picard iterates is a fixed point of the Bellman equation and inherits the same Lipschitz bound.\end{proof}

\section{Conclusion}

Lipschitz regularity is a key input for efficient numerical methods and for understanding how discretization, approximation, estimation, and learning errors propagate through dynamic programming. We establish explicit weighted Lipschitz bounds for Wasserstein-robust value functions in the kernel-robust and exogenous-noise formulations.

The main qualitative message is that Wasserstein robustness can regularize the value function. Example~\ref{exmpl:holder} shows that the non-robust and KL-robust value functions need not be Lipschitz even when the primitives are Lipschitz. In contrast, under the assumptions studied here, every positive Wasserstein radius yields a locally Lipschitz minimal Bellman fixed point. The proof leverages Wasserstein radius-sensitivity estimates and the Picard iteration arguments in Sections~\ref{section:kernel} and~\ref{section:caa_cau}.

\subsection*{Acknowledgments}
J. Blanchet gratefully acknowledges support from DoD through the grant ONR 1398311, also support from NSF via grants 2229012, 2312204, and 2403007 is gratefully acknowledged.

\bibliographystyle{apalike}
\bibliography{ref}

\end{document}